\documentclass[10pt]{article}

\usepackage{amsthm}
\usepackage{hyperref}
\usepackage{amssymb}
\usepackage{latexsym}
\usepackage{amsmath}
\usepackage{amsfonts}
\usepackage{version}
\usepackage[dvips]{graphicx}
\usepackage[LGR,T1]{fontenc}%
\usepackage[francais,english]{babel}
\usepackage{url}
\usepackage{enumerate}
\usepackage{hyperref}
\usepackage{color}
\usepackage{bm}
\usepackage{authblk}

\numberwithin{equation}{section}
\theoremstyle{plain}

\newtheorem{prop}{Proposition}

\numberwithin{lm}{section}

\numberwithin{term}{section}

\numberwithin{claim}{section}
\newtheorem{cor}{Corollary} 
\numberwithin{cor}{section}

\numberwithin{problem}{section}
\theoremstyle{definition} 
 
\numberwithin{definition}{section}
 
\numberwithin{example}{section}

\numberwithin{conjecture}{section}

\numberwithin{condition}{section}
\theoremstyle{remark} 

\numberwithin{remark}{section}

\numberwithin{remark}{section}

\vspace{1cm}

\date{August 15, 2024}

\begin{document}
\title{Catalan Numbers, Riccati Equations and Convergence}

\author[1]{Yicheng Feng} 
\author[2]{Jean-Pierre Fouque}
\author[3]{Tomoyuki Ichiba }
\affil[1,2,3]{University of California Santa Barbara}

\maketitle

\begin{abstract}
We analyze both finite and infinite systems of  Riccati equations derived from  stochastic differential games on infinite networks. We discuss a connection to the Catalan numbers and the convergence of the Catalan functions by Fourier transforms. 
\end{abstract}

\bigskip 

{\bf Keywords:} Catalan functions, Riccati equation for periodic network, Stochastic differential games for infinitely many players

\section{Introduction}
The Catalan numbers $\,C_{n}\,$, $\,n \ge 0 \,$ appear as a sequence of natural numbers defined by 
\begin{equation} \label{eq: Catalan}
C_{n} \, :=\,  \frac{\,1\,}{\,n+1\,} \left ( \begin{array}{c} 2n \\ n \end{array} \right ) \, =\,  \frac{\,(2n )! \,}{\,n ! \, (n+1)! \,} \, , \quad n \ge 0 \, .  
\end{equation}
For example, $\, C_{0} \, =\,  1\,$, $\, C_{1} \, =\,  1\,$, $\, C_{2}	 \, =\, 2\,$ and so on. This increasing sequence satisfies the recurrence relations 
\begin{equation} 
C_{n} \, =\,  C_{0} C_{n-1} + C_{1} C_{n-2} + \cdots + C_{n-1} C_{0} \, =\,   \sum_{j=1}^{n} C_{j-1} C_{n-j} \, , \quad n \ge 1 \, 
\end{equation}
and grows like $\, 4^{n} n^{-3/2} \, / \, \sqrt{\pi} \,$, as $\, n \to \infty\,$. 
The Catalan numbers appear in many combinatorial counting problems, for example, counting of non-crossing partitions, the number of the Dyck words, the number of standard Young tableaux (see the monographs \cite{MR0847717}, \cite{MR1676282}, \cite{MR3467982}  by Stanley).   

In this paper we shall discuss  the Catalan numbers and more generally Catalan functions in the context of the stochastic differential games on infinite network introduced in the recent papers (Feng, Fouque and Ichiba \cite{FFI-21a} and \cite{FFI-21b}, see also the referenced papers therein for the related mean-field games,  some topics of stochastic differential games and their applications), where the Catalan functions are defined by the solution to the system of the infinite Riccati equations. Note that the system of the infinite Riccati equations determines the Nash equilibrium of the stochastic differential game for infinitely many players. Then we prove the convergence of the solution of the finite Riccati equation corresponding to a stochastic differential game for finitely many players (say $\,N\,$ players) on a periodic network, as $\, N \to \infty\,$, to the solution of a system of infinite Riccati equations.

Following  Feng, Fouque and Ichiba \cite{FFI-21a}, let us recall the following  Riccati equation for the countably many continuous functions $\, \varphi^{i}_{t}\,$, $\, i \in \mathbb N_{0}\,$, $\, 0 \le t \le T\,$, given by the  system 
\begin{equation} \label{eq: Riccati inf}
\dot{ \varphi}^{\, i}_{t} \, =\, \frac{\, {\mathrm d} \varphi_{t}^{i}\,}{\,{\mathrm d} t \,}\, =\,  \sum_{j=0}^{i} \varphi_{t}^{\, j} \, {\varphi}^{\, i-j}_{t} - \varepsilon^{i} ; \quad i \, \in \mathbb N_{0} \, , 
\end{equation}
where $\, \varepsilon^{i}\,$ are given by some real constants $\, \varepsilon^{0} := \varepsilon\,$, $\, \varepsilon^{1} \, :=\, - \varepsilon \,$, $\, \varepsilon^{i} \, =\,  0 \,$ for $\, i \neq 0 , 1\,$, and the terminal conditions are $\, {\varphi}^{0}_{T} \, :=\, c  \,$, $\, {\varphi}^{1}_{T} := - c \,$, $\, {\varphi}_{T}^{i} \, :=\,  0 \,$ for $\, i \, \neq \, 0, 1 \,$. Here, ``$\, \dot{\quad}"$ denotes the differentiation with respect to $\, t \,$, and the superscript $\, i\, $ is not the power of function $\, \phi\, $ but the index $\, i \in \mathbb N_0\, $. Given $\, \varepsilon > 0 \,$ and $\, c \ge 0 \,$, the solution $\, \{\varphi_{t}^{i}, i \in \mathbb N, 0 \le t \le T\} \,$ of \eqref{eq: Riccati inf} exists and is unique (Lemma 1 of \cite{FFI-21a}). We call such sequence of functions the {\it Catalan functions}.

The solution $\, \varphi^{i}_{t}\,$, $\, 0 \le t\le T\,$, $\, i \in \mathbb N_{0}\,$ depends on $\, \varepsilon\,$ and $\, T\,$. Particularly, we take $\, \varepsilon \, =\,  1 \, =\,  \varepsilon^{0} \, =\,  - \varepsilon^{1}\,$, and consider the stationary solution by letting the time derivative zero, that is, $\, \dot{\varphi}^{\, i}_{t} \equiv 0 \,$, $\, i \in \mathbb N_{0}\,$, $\, t \ge 0 \,$. Then the stationary solution $\, \{\varphi^{i} \}_{i \in \mathbb N_{0}}\,$ of \eqref{eq: Riccati inf} satisfies 
\[
\varphi^{0} \, =\,  1 \, , \quad \varphi^{1} \, =\,  - \frac{\,1\,}{\,2\,} \, , \quad \text{ and } \quad \varphi^{i} \, =\, - \frac{\,1\,}{\,2\,} \sum_{j=1}^{i-1} \varphi^{j} \varphi^{i-j}\, ; \quad i \ge 2 \, .   
\]
Thus, the relation between the stationary solution $\, \{\varphi^{i} \}_{i \ge 1}\,$ of \eqref{eq: Riccati inf}  and the Catalan numbers $\, \{C_{i}\}_{i \in \mathbb N_{0}}\,$in \eqref{eq: Catalan} is 
\begin{equation}
\varphi^{i} \, =\,  - \frac{2 C_{i-1}}{4^{i}} \,; \quad i \ge 1 \, . 
\end{equation}

Let us also recall the Riccati equation for $\,N\,$ continuous functions $\, \phi^i_t\, $, $\, i = 0, 1, \ldots , N-1\, $, $\, 0 \le t \le T\, $, given by the following system 
\begin{equation} \label{eq: Riccati N}
\dot{\phi}^i_t := \, \frac{\,{\mathrm d} \phi_{t}^{i} \,}{\, {\mathrm d} t\,}=\, \sum_{j=0}^{N-1} \phi_t^j \phi_t^{N+i-j} - \varepsilon^i \, ; \quad  t \ge 0 
\end{equation}
of ordinary differential equations for $\, i = 0, 1, \ldots , N-1 \, $ and $\, 0 \le t \le T\, $ with the given terminal values $\, \phi_T^0 := c =: - \phi^1_T > 0 \, $, $\, \phi^i_T := 0 \, $, $\, i = 2, \ldots , N-1\, $ and real constants $\, \varepsilon^0 := \varepsilon =: -\varepsilon^1 > 0 \, $ and $\, \varepsilon^i :=0\, $ for $\, i = 2, \ldots , N-1\, $.  We impose the periodic condition $\, \phi_\cdot^i = \phi_\cdot^{i+N}\, $ for every $\, i \in \mathbb Z \, $. The solution $\, \{\phi_{t}^{i}, i = 0, 1, \ldots , N-1, 0 \le t \le T\}\,$ of \eqref{eq: Riccati N} exists uniquely and depends on $\,N\,$.

\medskip 

The finite system \eqref{eq: Riccati N} leads us to the Nash equilibrium for the $\,N\,$-player, linear-quadratic stochastic differential game on the finite directed chain periodic network, while the infinite system \eqref{eq: Riccati inf} leads us to the Nash equilibrium for the infinitely many player, linear-quadratic stochastic differential game on the infinite directed chain network. In \cite{FFI-21a} and \cite{FFI-21b} the question of the convergence of the Nash equilibrium for the $\,N\,$-player game to the Nash equilibrium for the infinitely many player game was left as an open question in the periodic case considered here. In this paper we solve this open question positively. 

The main results of this paper are the following propositions of convergence. 

\begin{prop}\label{prop: cvRiccati} For any fixed $\, j \in \mathbb N_{0}\,$ and $\, t \in [0, T]\,$, the solution $\, \phi_{t}^{j}  \,$ of the finite system \eqref{eq: Riccati N} converges to $\,\varphi_{t}^{j} \,$ of the infinite system \eqref{eq: Riccati inf}, as $\, N \to \infty\,$. That is, 
\begin{equation} \label{eq: prop:cvRiccati}
\lim_{N\to \infty} \phi_{t}^{j} \, =\,  \varphi_{t}^{j} \, ; \quad j \in \mathbb N_{0} \, , \, t \in [0, T ] \, .  
\end{equation}
\end{prop}

\begin{prop} \label{prop: cvprod} For any fixed $\, i \in \mathbb N_{0}\,$ and $\, t \in [0, T] \,$, we have the convergence results 
\begin{equation} \label{eq: cvprod}
\lim_{N\to \infty} \sum_{j=0}^{N-1} \phi_{t}^{j} \phi_{t}^{N+i-j} \, =\,  \sum_{j=0}^{i} \varphi^{j}_{t} \varphi^{i-j}_{t} \, , \quad \text{ and } \quad 
\lim_{N\to \infty} \sum_{j=i+1}^{N-1} \phi^{j}_{t} \phi^{N+i-j}_{t}\, =\,   0 \,. 
\end{equation}
\end{prop}

\begin{prop} \label{prop: unif conv} For any $\, K \in \mathbb N_{0}\,$, $\, T > 0 \,$, the solution $\, \{\phi_{t}^{i}, i = 0, 1, \ldots , N-1, 0 \le t \le T\}\,$ of \eqref{eq: Riccati N} and the solution $\, \{\varphi_{t}^{i}, i \in \mathbb N, 0 \le t \le T\} \,$ of \eqref{eq: Riccati inf} satisfy 
\begin{equation}
\lim_{N\to \infty} \sup_{0 \le i \le K} \sup_{0 \le t \le T} \lvert \phi^{i}_{t} - \varphi^{i}_{t} \rvert \, =\,  0 \, . 
\end{equation}
That is, the first $\,K\,$ elements of the solution of \eqref{eq: Riccati N} converges uniformly to the first $\,K\,$ elements of the solution of \eqref{eq: Riccati inf}, as $\, N \to \infty\,$. 
\end{prop}

These results are proved in the following sections by Fourier transforms. The key observations are the representations \eqref{eq: relation FT-NDFT} and \eqref{eq: phitj ID} of the solutions $\, \{\phi^{j}_{t}\}\,$ and $\, \{\varphi^{j}_{t}\}\,$ of the Riccati equations \eqref{eq: Riccati N} and \eqref{eq: Riccati inf} in terms of the solution $\, \{f_{t}(x)\}\,$ in \eqref{eq: ftx1} of an auxiliary Riccati equation \eqref{eq: Riccati FT} below. 

After this manuscript was prepared, the recent papers \cite{MR4621507} and \cite{MR4615116} by Miana and Romero were brought up to our attention.  In these papers a slightly general quadratic equation for Catalan generating functions, its spectrum and resolvent operator are studied from the point of view of functional analysis. In contrast to \cite{MR4621507} and \cite{MR4615116}, the results here on the convergence of the solutions are more concrete, because of the specific form \eqref{eq: Riccati inf} of quadratic equation and because of the Fourier transforms. The generalization of the results in the current paper will be a theme of another paper. 


\section{Fourier transforms and Riccati equations}

Let us define the discrete Fourier transform $\, \{\widehat{\phi}_{t}^{k} , k \, =\, 0, 1, \ldots , N-1\}\,$, $\, 0 \le t \le T\,$ of the solution $\, \{ \phi_{t}^{i}, i \, =\,  0, 1, \ldots , N-1, 0 \le t \le T\} \,$  of the Riccati equation (\ref{eq: Riccati N}) by 
\begin{equation} \label{eq: hatphitk}
\widehat{\phi}_{t}^{k} \, :=\,  \sum_{j=0}^{N-1} \phi^{j}_{t}  \exp \Big( - \frac{2\pi \sqrt{-1} \, j k }{N}  \Big) \, ; \quad k \, =\,  0, 1, \ldots , N-1 \, , 0 \le t \le T \, . 
\end{equation}
Here, the superscript $\,k\,$ for $\,\widehat{\phi}_{\cdot}\,$ is not the  power but the index. $\, \sqrt{-1}\,$ is the complex square root of $\, - 1\,$. Inverting the discrete Fourier transform, we obtain 
\begin{equation} \label{eq: DFT inversion}
\phi_{t}^{j} \, =\, \frac{\,1\,}{\,N\,} \sum_{k=0}^{N-1} \widehat{\phi}^{k}_{t} \exp \Big(  \frac{2\pi \sqrt{-1}\,  j k }{N}  \Big)  \, ; \quad j\, =\,  0, 1, \ldots , N-1 \, ,  
\end{equation}
and in particular, 
\begin{equation} \label{eq: DFT inversion 0}
\phi_{t}^{0} \, =\,  \frac{\,1\,}{\,N\,} \sum_{k=0}^{N-1} \widehat{\phi}^{k}_{t} \, ; \quad 0 \le t \le T \, . 
\end{equation}

Since the discrete Fourier transform of the convolution of two sequences is the product of their discrete Fourier transforms, it follows from the Riccati equation (\ref{eq: Riccati N}) that $\,\widehat{\phi}_{t}^{k}\,$ in \eqref{eq: hatphitk} satisfies the one-dimensional Riccati equation 
\begin{equation} \label{eq: Riccati NDFT}
\dot{\widehat{\phi}_{t}^{k}} \, =\,  (\widehat{\phi}_{t}^{k})^{2} - (1 - e^{- 2\pi \sqrt{-1} k/N}) \varepsilon \, ; \quad 0 \le t \le T \,  
\end{equation}
with the terminal condition $\, \widehat{\phi}_{T}^{k} \, =\,  ( 1- e^{-2\pi \sqrt{-1} k/N}) c \,$ for $\, k \, =\,  0 , 1, \ldots , N-1\,$. 

In a similar manner, replacing $\, k / N\,$ by $\,x\,$ in \eqref{eq: Riccati NDFT}, let us consider the following, one-dimensional, auxiliary Riccati equation for $\,\mathbb C\,$-valued function $\, \{f_{t}(x), 0 \le t \le T \, , \, x \in [0, 1]  \}\,$ defined by  
\begin{equation} \label{eq: Riccati FT}
\dot{f}_{t}(x) \, =\,  (f_{t}(x))^{2} - ( 1 - e^{-2\pi \sqrt{-1} x} ) \varepsilon \, ; \quad 0 \le t \le T \, , \quad x \in [0, 1] \, 
\end{equation}
with the terminal condition $\, f_{T}(x) \, =\,  ( 1 - e^{-2\pi \sqrt{-1} x})c \,$, $\, x \in [0, 1] \,$. 

Since both Riccati equations \eqref{eq: Riccati NDFT} and \eqref{eq: Riccati FT} are  scalar-valued ordinary differential equations, we solve them explicitly by the standard method of solving the general Riccati equation of the form  
\begin{equation} \label{eq: gRiccati1}
\dot{y}_{t} \, =\,  a_{t} + b_{t} y_{t} + c_{t} (y_{t})^{2} \, ; \quad 0 \le t \le T 
\end{equation}
with some (smooth) functions $\, a_{\cdot}, b_{\cdot}, c_{\cdot}  \,$. That is, solving a second-order ordinary differential equation 
\begin{equation} \label{eq: gRiccati2}
\ddot{u}_{t} - \Big( b_{t} + \frac{\,\dot{c}_{t}\,}{\,c_{t}\,}\Big) \dot{u}_{t} + a_{t} c_{t} u_{t} \, =\,  0 \, 
\end{equation}
for $\,\{ u_{t}\}\,$, we obtain the solution $\, y_{t} \, =\,  - \dot{u}_{t} / (c_{t} u_{t})\,$, $\, 0 \le t \le T\,$ for the general Riccati equation. 
The solutions to our Riccati equations \eqref{eq: Riccati NDFT} and \eqref{eq: Riccati FT} are given by the following proposition. 

\begin{prop} \label{prop: Riccati FT}
The solution of the auxiliary Riccati equation \eqref{eq: Riccati FT} is given by 
\begin{equation} \label{eq: ftx1}
f_{t}(x) \, =\, \sqrt{ \varepsilon} \,  \, {\mathfrak r} (x) \, e^{\sqrt{-1} {\bm \theta} (x) } \cdot \frac{\, {\mathfrak a}^{+} (x) {\mathfrak e}_{t}^{+}(x) - {\mathfrak a}^{-}(x){\mathfrak e}_{t}^{-}(x) \,}{\,{\mathfrak a}^{+}(x) {\mathfrak e}_{t}^{+}(x) + {\mathfrak a}^{-}(x){\mathfrak e}_{t}^{-}(x) \,} \, ,  
\end{equation}
where $\, {\mathfrak a}^{\pm}(x)  \,$ and $\, \mathfrak e_{t}^{\pm} (x) \,$ are $\, \mathbb C\,$-valued functions defined by 
\begin{equation} \label{eq: ftx2}
{\mathfrak a}^{\pm}(x) \, :=\,  \sqrt{\, \varepsilon \, } \pm c \, {\mathfrak r} (x) \, e^{\sqrt{-1} {\bm \theta} (x) }  \, , \quad 
\mathfrak e_{t}^{\pm} (x) \, :=\,  \exp \Big( \pm \sqrt{ \varepsilon\, } {\mathfrak r} (x) e^{\sqrt{-1} {\bm \theta}  (x) }\,  (T-t) \Big)  \, ; \quad 0 \le t \le T \, 
\end{equation}
with 
\begin{equation} \label{eq: ftx3}
{\mathfrak r} (x) \, :=\,  [ 2 ( 1 - \cos ( {\,2\pi x  \,}) ) ]^{1/4} \, , \quad {\bm \theta} (x)  \, :=\, \frac{\,1\,}{\,2\,}\arctan \Big( \frac{\,\sin ( 2 \pi x ) \,}{\,1 - \cos ( 2 \pi x) \,}\Big) \, \in [0, \pi) \, 
\end{equation}
for fixed $\, x \in [0, 1 ] \,$. 
\end{prop}

\begin{proof} For each fixed $\,x \in [0, 1 ] \,$, we shall solve the Riccati equation \eqref{eq: Riccati FT} for $\,\{f_{t}(x)\}\,$, as the special case of the general Riccati equation \eqref{eq: gRiccati1} with $\, a_{t} := - ( 1 - e^{-2\pi \sqrt{-1}x} ) \varepsilon\,$, $\, b_{t} \, :=\, 0 \,$, $\, c_{t} \, =\,  1\,$, $\, 0 \le t \le T\,$. By the transformation from $\,y_{\cdot}\,$ in \eqref{eq: gRiccati1} to $\, u_{\cdot}\,$ in \eqref{eq: gRiccati2}, it amounts to solving the second-order differential equation 
\[
\ddot{u}_{t} + (1 - e^{-2\pi \sqrt{-1}x} ) \varepsilon u_{t} \, =\,  0 \, ; \quad 0 \le t \le T \, . 
\]
With the definitions \eqref{eq: ftx3} of $\,{\mathfrak r}(x) \,$ and $\, {\bm \theta}(x) \,$,  the square roots of $\, - ( 1 - e^{-2\pi \sqrt{-1}x}) \,$ is given by $\, \pm \sqrt{-1} {\mathfrak r}(x) e^{\sqrt{-1} {\bm \theta}(x)}\,$. Hence, the solution $\, u_{\cdot}\,$ to the second-order differential equation is given by  
\[
u_{t} (x) \, =\,  {\mathfrak c}_{1}(x) \cdot  e^{\sqrt{-1}{\mathfrak r}(x) e^{\sqrt{-1} {\bm \theta}(x)} t } + {\mathfrak c}_{2} (x) \cdot e^{ - \sqrt{-1}{\mathfrak r}(x) e^{\sqrt{-1} {\bm \theta}(x)} t } \, ; \quad 0 \le t \le T 
\]
for some $\, {\mathfrak c}_{i}(x)\,$, $\, i \, =\,  1, 2\,$ which are determined by the terminal condition $\, f_{T}(x) \, =\, - \dot{u}_{T} (x) \, /\,  u_{T} (x) \,$, and $\, f_{t}(x) \, =\, - \dot{u}_{t}(x)  \, /\,  u_{t}(x)  \,$ is given by \eqref{eq: ftx1} for $\, x \in [0, 1] \,$, $\, t \in [0, T] \,$. 
\end{proof}

\begin{prop}
With $\, \{f_{t}(x)\}\,$ defined in \eqref{eq: ftx1}, the solution of the Riccati equation \eqref{eq: Riccati NDFT} and the solution of the Riccati equation \eqref{eq: Riccati N}  are represented by 
\begin{equation} \label{eq: relation FT-NDFT}
\widehat{\phi}_{t}^{k} \, =\, f_{t} \Big( \frac{\,k\,}{\,N\,}\Big)  \, , \quad \text{ and } \quad 
{\phi}_{t}^{k} \, =\, \frac{\,1\,}{\,N\,}\sum_{j=1}^{N} f_{t} \Big( \frac{\,k\,}{\,N\,}\Big) \exp \Big( {2\pi} \sqrt{-1} j \cdot \frac{\,k\,}{\,N\,} \Big) \, 
\end{equation}
for $\, k \, =\, 0, 1, \ldots , N-1 \, $, $\, 0 \le t \le T \,$.  
Thus, there exists a constant $\,c_{T} := \sup_{0 \le t \le T}\sup_{x \in [0, 1]} \lvert f_{t}(x) \rvert  \in (0, \infty) \,$,  such that 
\begin{equation} \label{eq: bound for phi hat}
\sup_{N \ge 2} \sup_{0 \le k \le N-1} \sup_{0 \le t \le T} \lvert {\phi}^{k}_{t} \rvert \, \le\, \sup_{N \ge 2} \sup_{0 \le k \le N-1} \sup_{0 \le t \le T} \lvert \widehat{\phi}^{k}_{t}\rvert \le c_{T} \, . 
\end{equation}
\end{prop}

\begin{proof}
For each fixed $\,k \, =\,  0, 1, \ldots , N-1\,$,  we solve the Riccati equation \eqref{eq: Riccati NDFT} for the discrete Fourier transform $\, \widehat{\phi}^{k}_{t}\,$ and obtain $\,\widehat{\phi}_{t}^{k} \, =\,  f_{t}(k / N) \,$ in a similar procedure, replacing $\, k / N\,$ by $\,x\,$ in the proof of Proposition \ref{prop: Riccati FT}.  Substituting it to the inverse discrete Fourier transform \eqref{eq: DFT inversion}, we obtain \eqref{eq: relation FT-NDFT}. The uniform bound \eqref{eq: bound for phi hat} is obtained directly by the representations \eqref{eq: relation FT-NDFT}. 
\end{proof}

In order to prove Proposition \ref{prop: cvRiccati}, we derive the following representation of the infinite Riccati solution $\, \{\varphi^{k}_{t}\}\,$ in terms of the auxiliary Riccati solution $\, \{f_{t}(x)\}\,$ in \eqref{eq: ftx1}. 

\begin{prop} \label{prop: phitj ID} With the solution $\, \{ f_{t}(x)\} \,$ in \eqref{eq: ftx1}  of the auxiliary Riccati equation \eqref{eq: Riccati FT}, the solution $\, \{ \varphi^{j}_{t}\} \,$ of the infinite Riccati equation \eqref{eq: Riccati inf} is represented as  
\begin{equation}  \label{eq: phitj ID} 
\varphi^{j}_{t} \, =\,  \int^{1}_{0} f_{t}(x) e^{2\pi \sqrt{-1} j x} {\mathrm d} x \, ; \quad j \in \mathbb N_{0} \, , \, \, 0 \le t \le T\, . 
\end{equation}
Consequently, we have the upper bound 
\begin{equation} \label{eq: upper bound varphitj}
\sup_{j \in \mathbb N_{0}} \sup_{0 \le t \le T}  \lvert \varphi_{t}^{j}\rvert \le c_{t} \, =\,  \sup_{0 \le t \le T} \sup_{x \in [0, 1]} \lvert f_{t}(x) \rvert \in (0, \infty ) \, . 
\end{equation}
\end{prop}

\begin{proof}Note that the family $\, \{ e^{-2\pi \sqrt{-1} j x } , j \in \mathbb N_{0}\} \, $ of continuous functions on $\, [0, 1]\,$ forms an orthonormal basis of the space $\,L^{2}([0, 1]) \,$, and the right hand of \eqref{eq: phitj ID}  is the $\,j\,$-th Fourier coefficient of $\, f_{t} \,$ with respect to this orthonormal basis, that is, 
\begin{equation} \label{eq: phitjFourierS}
f_{t}(x) \, =\,  \sum_{j=0}^{\infty} {\bf c}_{j, t} e^{- 2\pi \sqrt{-1}j x} \, , \quad {\bf c}_{j, t} \, :=\, \int^{1}_{0} f_{t}(y) e^{2\pi \sqrt{-1} j y} {\mathrm d} y  \, ; \quad x \in [0, 1 ] \, , t \in [0,T ] \, . 
\end{equation}

To show \eqref{eq: phitj ID}, we shall verify that the  Fourier coefficients $\, \{ {\bf c}_{j,t}\}\,$ satisfy the infinite Riccati equation \eqref{eq: Riccati inf} and we apply its uniqueness of the solution. Since $\, \{ f_{t}(x)\}\,$ satisfies the auxiliary Riccati equation \eqref{eq: Riccati FT}, by the direct calculation we obtain 
\begin{equation} \label{eq: derivativephijt}
\begin{split}
\frac{\,{\mathrm d} \,}{\,{\mathrm d} t \,} \int^{1}_{0} f_{t}(x) e^{2\pi \sqrt{-1} j x } {\mathrm d} x \, & =\,  \int^{1}_{0} \dot{f}_{t}(x) e^{2\pi \sqrt{-1} j x } {\mathrm d} x \\
\, & =\,  \int^{1}_{0}  ( (f_{t}(x))^{2} - ( 1 - e^{-2\pi \sqrt{-1} x} ) \varepsilon  ) e^{2\pi \sqrt{-1} j x } {\mathrm d} x \\
 \, & =\,  \int^{1}_{0} (f_{t}(x))^{2} e^{2\pi \sqrt{-1} j x } {\mathrm d} x  - \varepsilon  \int^{1}_{0}( 1 - e^{-2\pi \sqrt{-1} x} )   e^{2\pi \sqrt{-1} j x } {\mathrm d} x \\
 \, & =\, \int^{1}_{0} (f_{t}(x))^{2} e^{2\pi \sqrt{-1} j x } {\mathrm d} x  - \varepsilon^{j} \, ,\quad j \in \mathbb N_{0} \, , \, \, t \in [0, T ] \, , 
\end{split}	
\end{equation}
where $\, \{ \varepsilon^{j} \}\,$ was defined as $\, \varepsilon^{0} \, =\,  \varepsilon = - \varepsilon^{1} > 0 \,$, and $\, \varepsilon^{i} \, =\,  0 \,$, $\, i \ge 2 \,$. For the first term of the right hand, it follows from \eqref{eq: phitjFourierS} and the convolution of the Fourier series that 
\begin{equation} \label{eq: derivativephijt2}
\begin{split}
\int^{1}_{0} (f_{t}(x))^{2} e^{2\pi\sqrt{-1}jx} {\mathrm d} x \, & =\,  \int^{1}_{0} \big( \sum_{\ell=0}^{\infty} {\bf c}_{\ell,t} e^{-2\pi \sqrt{-1}\ell x} \sum_{k=0}^{\infty} {\bf c}_{k,t} e^{-2\pi \sqrt{-1}k x} \big) e^{2\pi \sqrt{-1}j x} {\mathrm d} x 
\\
\, & =\,  \int^{1}_{0} \big( \sum_{k=0}^{\infty} {\bf b}_{k,t} e^{-2\pi \sqrt{-1}k x} \big) e^{2\pi \sqrt{-1}jx} {\mathrm d} x \, =\,  {\bf b}_{j,t} \, :=\,  \sum_{k=0}^{j} {\bf c}_{k,t} {\bf c}_{j-k, t} 
\\
\, & =\,  \sum_{k=0}^{j} \Big(  \int^{1}_{0} f_{t}(x) e^{2\pi \sqrt{-1}k x} {\mathrm d} x \Big) \Big( \int^{1}_{0} f_{t}(x) e^{2\pi \sqrt{-1}(j-k) x} {\mathrm d} x \Big) \, . 
\end{split}
\end{equation} 

Substituting this expression in \eqref{eq: derivativephijt}, and because of \eqref{eq: phitjFourierS}, we obtain the infinite Riccati equation 
\begin{equation}
\begin{split}
\dot{\bf c}_{j,t} \, &=\, \frac{\,{\mathrm d} \,}{\,{\mathrm d} t \,} \int^{1}_{0} f_{t}(x) e^{2\pi \sqrt{-1} j x } {\mathrm d} x \\
& \, =\,   \sum_{k=0}^{j} \Big(  \int^{1}_{0} f_{t}(x) e^{2\pi \sqrt{-1}k x} {\mathrm d} x  \Big) \Big( \int^{1}_{0} f_{t}(x) e^{2\pi \sqrt{-1}(j-k) x} {\mathrm d} x \Big) - \varepsilon^{j} \\
\, & =\, \sum_{k=0}^{j} {\bf c}_{k,t} {\bf c}_{j-k, t} - \varepsilon^{j} \, ; \quad j \in \mathbb N_{0} \, , 0 \le t \le T \, , 
\end{split}
\end{equation}
equivalent to \eqref{eq: Riccati inf}. Also,  the terminal condition is satisfied 
\[
{\bf c}_{T,j} \, =\, \int^{1}_{0} f_{T}(x) e^{2\pi \sqrt{-1}j x} {\mathrm d} x \, =\,  \int^{1}_{0} c ( 1 - e^{-2\pi \sqrt{-1}x})  e^{2\pi \sqrt{-1}j x} {\mathrm d} x \, =\,  c^{j} \, , 
\]
where $\, \{ c^{j}\} \,$ was defined as $\, c^{0} \, =\,  c \, =\,  - c^{1} > 0 \,$ and $\, c^{i} \, =\,  0 \,$, $\, i \ge 2 \,$. Thus, by the uniqueness of the solution to the infinite Riccati equation \eqref{eq: Riccati inf}, we identify $\,{\bf c}_{j, t} \, =\,  \varphi^{j}_{t}\,$, $\, j \in \mathbb N_{0}\,$, $\, t \in [0, T]\,$ as in \eqref{eq: phitj ID}. 
\end{proof}

\subsection{Proof of Proposition \ref{prop: cvRiccati}}
Now we shall prove Proposition \ref{prop: cvRiccati}. 
Substituting \eqref{eq: relation FT-NDFT} into the inverse discrete Fourier transform \eqref{eq: DFT inversion}, we obtain the Riemann sum 
\[
\phi_{t}^{j} \, =\,  \frac{\,1\,}{\,N\,} \sum_{k=0}^{N-1} \widehat{\phi}^{k}_{t} \exp \Big(  \frac{2\pi \sqrt{-1} j k }{N}  \Big) \, =\,  \frac{\,1\,}{\,N\,} \sum_{k=0}^{N-1} f_{t} \Big( \frac{\,k\,}{\,N\,}\Big)  \exp \Big(  2\pi \sqrt{-1}j \cdot \frac{ k }{N}  \Big) \, 
 \]
for $\,j \, =\,  0, 1, \ldots , N-1 \, , \, 0 \le t \le T  
\,$. Since $\, f_{t}(x) e^{2\pi \sqrt{-1} k x}\,$ is a continuous function of $\, x\,$ for every fixed $\, j \,$ and $\, t \,$, taking the limit as $\, N \to \infty\,$, we obtain the limit of $\, \phi_{t}^{j}\,$, 
\begin{equation} 
\lim_{N\to \infty} \phi_{t}^{j} \, =\, \lim_{N\to \infty }  \frac{\,1\,}{\,N\,} \sum_{k=0}^{N-1} f_{t} \Big( \frac{\,k\,}{\,N\,}\Big)  \exp \Big(  2\pi \sqrt{-1}j \cdot \frac{ k }{N}  \Big) \, =\,  \int^{1}_{0} f_{t}(x) e^{2\pi \sqrt{-1}j x } {\mathrm d} x  \, =\,  \varphi^{j}_{t}
\end{equation}
for each fixed $\, j \in \mathbb N_{0} \,$ and $\, t \in [0, T ]\,$, thanks to the identification in Proposition \ref{prop: phitj ID}. 
\hfill $\,\square\,$ 

\subsection{Proof of Proposition \ref{prop: cvprod}}
The first part of the convergence results \eqref{eq: cvprod} is obtained in a similar manner as in the proof of Proposition \ref{prop: cvRiccati}. Indeed, using \eqref{eq: DFT inversion} and \eqref{eq: relation FT-NDFT}, we rewrite the sum as a Riemann sum, and then we take the limit, as $\, N \to \infty\,$, 
\begin{equation} \label{eq: prop: cvprod2}
\begin{split}
\sum_{j=0}^{N-1} \phi^{j}_{t} \phi^{N+i-j}_{t} \, & =\,  \sum_{j=0}^{N-1} \frac{\,1\,}{\,N\,} \sum_{k=0}^{N-1} \widehat{\phi}^{k}_{t} e^{2\pi \sqrt{-1}kj/N} \cdot \frac{\,1\,}{\,N\,} \sum_{\ell=0}^{N-1} \widehat{\phi}^{\ell}_{t} e^{2\pi \sqrt{-1}(N+i-j)\ell/N} 
\\
\, & =\,  \frac{\,1\,}{\,N^{2}\,} \sum_{k, \ell=0}^{N-1}  f_{t} \Big( \frac{\,k\,}{\,N\,} \Big) f_{t} \Big( \frac{\,\ell\,}{\,N\,}\Big) \sum_{j=0}^{N-1} e^{2\pi \sqrt{-1}(k-\ell)j/N} \cdot e^{2\pi \sqrt{-1}i \ell / N} 
\\
\, & =\, \frac{\,1\,}{\,N^{2}\,} \sum_{k, \ell=0}^{N-1}  f_{t} \Big( \frac{\,k\,}{\,N\,} \Big) f_{t} \Big( \frac{\,\ell\,}{\,N\,}\Big) \cdot N \cdot {\bf 1}_{\{k \, =\,  \ell\}} \cdot e^{2\pi \sqrt{-1}i \ell / N} \\
\, & =\, \frac{\,1\,}{\,N\,} \sum_{k=0}^{N-1} \Big[ f_{t} \Big( \frac{\,k\,}{\,N\,} \Big) \Big]^{2}e^{2\pi \sqrt{-1}i \ell / N}\\
 & \xrightarrow[N\to \infty]{} \int^{1}_{0} (f_{t}(x))^{2} e^{2\pi \sqrt{-1} i x } {\mathrm d} x \, =\, \sum_{j=0}^{i} {\bf c}_{j,t} {\bf c}_{i-j, t} \, =\,  \sum_{j=0}^{i} \varphi^{j}_{t} \varphi^{i-j}_{t}\, 
\end{split}
\end{equation}
for every $\, t \in [0, T]\,$ and $\, i \ge 0 \,$, because of \eqref{eq: phitj ID} and \eqref{eq: derivativephijt2}. Here, $\, {\bf 1}_{\{k=\ell\}}\,$ is the indicator function which takes $\,1\,$ on the set $\, k \, =\,  \ell\,$ and $\,0\,$, otherwise, and $\, {\bf c}_{\cdot,t}\,$ was defined in \eqref{eq: phitjFourierS}. This proves the first part of the convergence results \eqref{eq: cvprod}. 

For the second part of the convergence results, combining the first part \eqref{eq: prop: cvprod2} with the convergence of $\, \{\phi_{t}^{i}\}\,$ in Proposition \ref{prop: cvRiccati}, we obtain 
\begin{equation}
\sum_{j=i+1}^{N-1} \phi_{t}^{j} \phi_{t}^{N+i-j} \, =\,  \sum_{j=0}^{N-1} \phi^{j}_{t} \phi_{t}^{N+i-j} - \sum_{j=0}^{i} \phi_{t}^{j} \phi_{t}^{N+i-j} \xrightarrow[N\to \infty]{} \sum_{j=0}^{i} \varphi^{j}_{t} \varphi_{t}^{i-j} - \sum_{j=0}^{i} \varphi^{j}_{t} \varphi^{i-j}_{t} \, =\,  0 \, . 
\end{equation}

Therefore, we conclude the proof of Proposition \ref{prop: cvprod}. 
\hfill $\,\square\,$

\subsection{Proof of Proposition \ref{prop: unif conv}}
We shall evaluate the difference $\, D_{N}(t) := \sup_{0 \le i \le K} \sup_{0 \le s \le t} \lvert \phi_{s}^{i} - \varphi_{s}^{i}\rvert\,$, $\, 0 \le t \le T \,$. With the time-reversal $\, \overline{\phi}_{t}^{i} \, :=\,\phi_{T-t}^{i} \,$, $\, \overline{\varphi}_{t}^{i} \, :=\, \varphi^{i}_{T-t}\,$, $\, 0 \le t \le T\,$, it follows from the Riccati equations that for $\, i \, =\,  0, 1, \ldots , N-2\,$, $\, 0 \le t \le T\,$, 
\begin{equation*}
\begin{split}
- \dot{ \overline{\phi}}^{i}_{t} + \dot{ \overline{\varphi}}^{i}_{t} \, & =\,  \dot{\phi}_{t}^{i} - \dot{\varphi}_{t} \, =\,  \sum_{j=0}^{N-1} \phi_{t}^{j} \phi_{t}^{N+i-j} - \sum_{j=0}^{i} \varphi_{t}^{j} \varphi_{t}^{i-j} \\
\, & =\,  \sum_{j=i+1}^{N-1} \phi^{j}_{t} \phi_{t}^{N+i-j} + \sum_{j=0}^{i} [ ( \phi^{j}_{t} - \varphi^{j}_{t}) \phi_{t}^{i-j} + \varphi^{j}_{t} (\phi_{t}^{i-j} - \varphi^{i-j}_{t}) ] \\
\, & =\,  \sum_{j=i+1}^{N-1} \overline{ \phi}^{j}_{t} \overline{ \phi_{t}}^{N+i-j} + \sum_{j=0}^{i} [ ( \overline{ \phi}^{j}_{t} - \overline{ \varphi}^{j}_{t}) \overline{ \phi}_{t}^{i-j} + \overline{ \varphi}^{j}_{t} (\overline{ \phi}_{t}^{i-j} - \overline{ \varphi}^{i-j}_{t}) ] \, . 
\end{split}
\end{equation*}
Since we have $\, \overline{\phi}_{0}^{i} \, =\,  \phi_{T}^{i} \, =\,  \varphi_{T}^{i} \, =\,  \overline{\varphi}^{i}_{0}\,$, integrating both sides over $\, [ 0, s] ( \subseteq [0, T] ) \,$, taking the absolute values and using the triangle inequality, we obtain 
\begin{equation}
\begin{split}
\lvert \overline{\phi}^{i}_{s} - \overline{\varphi}^{i}_{s}\rvert \le \int^{s}_{0} \lvert  \sum_{j=i+1}^{N-1} \overline{ \phi}^{j}_{u} \overline{ \phi}_{u}^{N+i-j} \rvert {\mathrm d} u + \int^{s}_{0} 
\sum_{j=0}^{i} [ \lvert  \overline{ \phi}^{j}_{u} - \overline{ \varphi}^{j}_{u} \rvert \cdot \lvert 	  \overline{ \phi}_{u}^{i-j}\rvert  + \lvert \overline{ \varphi}^{j}_{u} \rvert \cdot  \lvert \overline{ \phi}_{u}^{i-j} - \overline{ \varphi}^{i-j}_{u}\rvert ] {\mathrm d} u 
\end{split}
\end{equation}

Then the difference $\, D_{N}(t) \,$ satisfies the inequality 
\begin{equation}
\begin{split}
D_{N}(t) \, & =\,  \sup_{0 \le i \le K} \sup_{0 \le s \le t} \lvert \phi_{s}^{i} - \varphi_{s}^{i} \rvert \, =\,  \sup_{0 \le i \le K} \sup_{0 \le s \le t} \lvert \overline{\phi}_{s}^{i} - \overline{\varphi}_{s}^{i} \rvert\\
& \le \int^{t}_{0} \sup_{0 \le i \le K} \sup_{0 \le u \le s} \lvert \sum_{j=i+1}^{N-1} \overline{ \phi}^{j}_{u} \overline{ \phi}_{u}^{N+i-j} \rvert {\mathrm d} u \\
& \quad {} + \int^{t}_{0} \sup_{0 \le i \le K} \sup_{0 \le u \le s} \max ( \lvert \overline{\phi}_{u}^{i} \rvert\, , \lvert \overline{\varphi}_{u}^{i} \rvert ) D_{N}(s) {\mathrm d} s \\
& \le \,  c_{N,1}(t) + \int^{t}_{0} c_{N,2}(s) D_{N}(s) {\mathrm d} s \, ,  
\end{split}
\end{equation}
where we defined 
\[
c_{N,1}(t) \, :=\, t \cdot \sup_{0 \le i \le K} \sup_{T-t \le u \le T} \lvert \sum_{j=i+1}^{N-1} \phi^{j}_{u} \phi_{u}^{N+i-j} \rvert \le c_{N, 1}(T) \, , 
\]
\[
c_{N, 2} (t) \, :=\,  K \cdot \sup_{0 \le i \le K}\sup_{T-t \le u \le T} \max ( \lvert {\phi}_{u}^{i} \rvert\, , \lvert {\varphi}_{u}^{i}\rvert ) \le c_{N, 2}(T) < \infty \, 
\]
for $\, 0 \le t \le T \,$. Note that by \eqref{eq: bound for phi hat} and \eqref{eq: upper bound varphitj}, we have $\, \sup_{N} c_{N,2}(T) < \infty\,$. Applying the Gronwall inequality, we obtain 
\begin{equation} \label{eq: Gronwall0}
D_{N}(T) \le c_{N, 1}(T) \exp \Big( \int^{T}_{0} c_{N, 2}(t) {\mathrm d} t \Big) \, . 
\end{equation}

Since the function $\,f_{\cdot}(\cdot)\,$ is bounded, we may refine the proof of Propositions \ref{prop: cvRiccati}-\ref{prop: cvprod}. Particularly, the approximation of the Riemann sum in \eqref{eq: prop: cvprod2} is uniform over $\, i = 0, 1, \ldots , K\,$ and over $\, [0, T]\,$. Thus, we obtain 
\[
\lim_{N\to \infty} c_{N,1}(T) \, =\,  \lim_{N\to \infty} T\cdot  \sup_{0\le i \le K} \sup_{0 \le u \le T} \lvert \sum_{j=i+1}^{N-1} \phi^{j}_{u} \phi_{u}^{N+i-j} \rvert \, =\,  0 \, . 
\] 
Therefore, combining this with \eqref{eq: Gronwall0}, we conclude the proof of Proposition \ref{prop: unif conv}: 
\[
\lim_{N\to \infty} \sup_{0 \le i \le K} \sup_{0 \le t \le T} \lvert \phi^{i}_{t} - \varphi^{i}_{t} \rvert \, =\,  \lim_{N\to \infty} D_{N}(T) \, \le \lim_{N\to \infty} c_{N,1}(T) \exp \Big( \int^{T}_{0} c_{N,2}(t) {\mathrm d} t \Big) \, =\,   0 \, . 
\]
\hfill $\,\square\,$

As a consequence of Proposition \ref{prop: unif conv}, we have the following corollary which resolves the open question left in \cite{FFI-21a}. 
\begin{cor} 
The $\,N\,$-player Nash equilibrium of linear quadratic stochastic differential games on the directed chain periodic network in \cite{FFI-21a} converges to the infinitely many player Nash equilibrium of linear quadratic stochastic differential games on the infinite directed chain network in \cite{FFI-21a}. 
\end{cor}

\section*{Acknowledgement}
{Part of research was supported by National Science Foundation NSF DMS-2008427. T.I.  would like to thank the Isaac Newton Institute for Mathematical Sciences, Cambridge, for support and hospitality during the programme Stochastic Systems for Anomalous Diffusions (SSD), supported by EPSRC grant no EP/K032208/1, where a part of work on this paper was undertaken.}


\section{Appendix}
\subsection{Finite system solved by matrix Riccati equation}

The above Riccati equation (\ref{eq: Riccati N}) can be written as a matrix Riccati equation  
\begin{equation} \label{eq: mRiccati N}
\dot{\Phi}(t) = \Phi(t) \Phi(t) - {\mathbf E} \, , \quad \Phi(T) := {\mathbf C} \, , 
\end{equation}
where $\, \Phi(\cdot)\, $ is the $\, N\times N\, $ matrix-valued function $\, \Phi(t) := (\Phi_{i,j} (t) )_{0\le i,j \le N-1} \, $, $\, 0 \le t \le T\, $ with $\, \Phi _{i,j}(t) := \phi_t^{i-j}\, $ for $\, 0 \le i, j \le N-1\, $ with the condition $\, \phi_\cdot^i = \phi_\cdot^{i+N}\, $ for every $\, i \in \mathbb Z \, $ and  $\, \mathbf E\, $ is an $\, N\times N\, $ matrix given by 

$$
\Phi (t) := \left( \begin{array}{ccccc}
\phi^0_t & \phi^{N-1}_t & \cdots & & \phi^1_t \\
\phi^1_t & \phi^0_t & \ddots & & \phi^2_t \\
\vdots & \ddots & \ddots & \ddots & \vdots \\
\vdots & \ddots & \ddots & \ddots & \phi^{N-1}_t\\
\phi^{N-1}_t & \cdots & & \phi^1_t & \phi^0_t \\
\end{array}
\right) \, , \quad 
\mathbf E := \left ( \begin{array}{ccccc} \varepsilon & 0 & \cdots & 0 & - \varepsilon \\
-\varepsilon & \varepsilon & \ddots & \ddots & 0 \\
0 & -\varepsilon & \ddots & \ddots  &\vdots \\
\vdots & \ddots & \ddots & \ddots & 0 \\
0& \cdots & 0 & -\varepsilon &  \varepsilon \\ 
\end{array}\right) \, , 
$$
and the $\, N\times N\, $ matrix $\, \mathbf C\, $ determines the terminal condition 
$$
\mathbf C := \left ( \begin{array}{ccccc} c & 0 & \cdots & 0 & -c \\
-c & c & \ddots & \ddots & 0 \\
0 & -c & \ddots & \ddots  &\vdots \\
\vdots & \ddots & \ddots & \ddots & 0 \\
0& \cdots & 0 & -c &  c \\ 
\end{array}\right) \, .  
$$
Here, $\, \dot{\Phi}(t) \,$ stands for the element wise differentiation of $\,\Phi(t)\,$ with respect to $\, t\,$.

Let us consider the time reversal parametrized by $\, \tau := T-t\, $ and $\, \Psi (\tau) := \Phi (T-\tau) \, $, $\, 0 \le t \le T\, $, $\, 0 \le \tau \le T\, $. Then the matrix-valued Riccati equation is 
\begin{equation}
\dot{\Psi} (\tau) = - \Psi(\tau) \Psi (\tau) + {\mathbf E} 
\end{equation}
for $\, 0 \le \tau \le T\, $ with the initial value $\, \Psi (0) := {\mathbf C}\, $. Its solution is given by 
\begin{equation} \label{eq: Riccati N sol}
\Psi (\tau) = ({\mathbf O}_{21} (\tau) + {\mathbf O}_{22} (\tau) {\mathbf C}) ({\mathbf O}_{11} (\tau) + {\mathbf O}_{12} (\tau) {\mathbf C})^{-1} \, , 
\end{equation}
where $\, {\mathbf O}_{ij}(\cdot)\, $, $\, 1\le i,j \le 2\, $ are the $\, N\times N\, $ block matrix elements of $\,{\mathbf O} (\cdot)\,$ defined by 
\begin{equation} \label{eq: M and O}
{\mathbf M} := \left( \begin{array}{cc} {\mathbf 0} &  {\mathbf I} \\ 
{\mathbf E} & {\mathbf 0} \end{array} \right) \, , \quad 
{\mathbf O}(\tau) := \left ( \begin{array}{cc} {\mathbf O}_{11}(\tau) & {\mathbf O}_{12}(\tau) \\ {\mathbf O}_{21}(\tau) & {\mathbf O}_{22}(\tau) \end{array} \right) := {\bf \exp} ( {\mathbf M}\tau) \, , 
\end{equation}
for $\, 0 \le \tau \le T\, $. Here $\, {\mathbf 0}\,  $ is $\, N\times N\, $ zero matrix and $\, {\mathbf I}\, $ is $\, N\times N\, $ identity matrix. Thus, we obtain the solution to the Riccati equation (\ref{eq: Riccati N}) as the first column of $\, \Phi(t) = \Psi(T-t)\, $ for $\, 0 \le t \le T\, $. 

\bigskip 

The characteristic polynomial of the $\,2N\times 2N\,$ matrix $\,{\mathbf M} \,$ in (\ref{eq: M and O}), in terms of $\, \lambda \in \mathbb C\,$, is simply given by 
\begin{equation}
\text{det} ( \lambda\, {\mathbf I} - {\mathbf M} ) \, =\,  (\lambda^{2} - \varepsilon)^{N} - (-\varepsilon)^{N} , 
\end{equation}
and hence the eigenvalues are 
\[
\lambda \, =\,  \pm  \sqrt{{\varepsilon} \cdot \Big( 1 - \exp \Big( \sqrt{-1} \cdot \frac{\,2\pi \, k\,}{\,N\,}\Big) \Big)} \, ; \quad k \, =\, 0, 1, \ldots , N-1 \, , 
\]
and $\, \lambda \, =\,  0 \,$ has multiplicity of $\, 2\,$. Thus, the size of the eigenvalues is bounded by $\, \sqrt{2 \varepsilon} \,$. For example, in the case of $\,N \, =\, 4\,$, the eight eigenvalues are 
\[
\{0, 0, \pm \sqrt{(1 + \sqrt{-1} )\varepsilon\, } \, , \pm \sqrt{(1 - \sqrt{-1} )\varepsilon\, },  \pm \sqrt{2\varepsilon\, } \} \, . 
\]

The direct numerical calculation of (\ref{eq: Riccati N sol}) is not stable for a large $\,\tau\,$, because of multiple eigenvalues. It is often suggested (e.g., Vaughan \cite{Vaughan}) to calculate iteratively   
\[
{\Psi} ((k+1) \Delta \tau) = ({\mathbf O}_{21} (\Delta \tau) + {\mathbf O}_{22} (\Delta\tau) {\Psi} (k \Delta \tau)) ({\mathbf O}_{11} (\Delta \tau) + {\mathbf O}_{12} (\Delta \tau) {\Psi} (k \Delta) )^{-1}  \, ; \quad k \, =\,  0, 1, 2, \ldots 
\]
with $\, \Psi (0) \, =\,  {\mathbf C} \,$, where $\, \Delta \tau \,$ is set to be small.

\subsection{Generating function for infinite Riccati equation} 
For the infinite system \eqref{eq: Riccati inf} let us recall the generating function $\, S_{t}(z) \, :=\, \sum_{k=0}^{\infty} z^{k} \varphi_{t}^{k}  \,$ for $\, \varphi_{\cdot}^{k}\,$, $\, k \, =\,  0, 1, 2, \ldots \,$ satisfies the scaler Riccati equation 
\[
\frac{\,{\mathrm d} \,}{\,{\mathrm d} t\,}S_{t}(z) \, =\, [ S_{t}(z)]^{2} - \varepsilon (1 - z) \, , \quad 0 \le t \le T \, , \quad S_{T}(z) \, =\,  c(1 - z) 
\]
for $\, \lvert z \rvert < 1 \,$. As in Proposition \ref{prop: Riccati FT}, the solution to this Riccati equation is given by 
\[
S_{t}(z) \, =\, \sqrt{\varepsilon ( 1 -z )} \cdot \frac{\,\overline{\mathfrak a}_{}^{+} \overline{\mathfrak e}_{t}^{+} - \overline{\mathfrak a}_{}^{-} \overline{\mathfrak e}_{t}^{-}\,}{\,\overline{\mathfrak a}_{}^{+} \overline{\mathfrak e}_{t}^{+} + \overline{\mathfrak a}_{}^{-} \overline{\mathfrak e}_{t}^{-}\,} \, , 
\]
where 
\[
\overline{\mathfrak a}_{}^{\pm} \, := \, \sqrt{\varepsilon(1 - z)} \pm c ( 1 - z)  \, , \quad   \overline{\mathfrak e}_{t}^{\pm} \, :=\, \exp \Big( \pm \sqrt{ \varepsilon ( 1 - z) } (T-t) \Big)
\]
for $\, 0 \le t \le T\,$, $\, \lvert z \rvert < 1 \,$.


\end{document}